\documentclass[11pt]{article}
\usepackage[a4paper]{anysize}\marginsize{3.5cm}{3.5cm}{1.3cm}{2cm}
\pdfpagewidth=\paperwidth \pdfpageheight=\paperheight
\usepackage{amsfonts,amssymb,amsthm,amsmath,eucal}
\usepackage{pgf}
\usepackage{bbm}
\theoremstyle{plain}
\newtheorem{thm}{Theorem}[section]
\newtheorem{theorem}[thm]{Theorem}

\newtheorem{lemma}[thm]{Lemma}
\newtheorem{proposition}[thm]{Proposition}
\newtheorem{corollary}[thm]{Corollary}

\theoremstyle{definition}

\newtheorem{remark}[thm]{Remark}
\newtheorem{example}[thm]{Example}

\newtheorem{thevarthm}[thm]{\varthmname}
\newenvironment{varthm}[1]{\def\varthmname{#1}\begin{thevarthm}}{\end{thevarthm}\def\varthmname{}}
\newenvironment{varthm*}[1]{\trivlist\item[]{\bf #1.}\it}{\endtrivlist}

\renewcommand\geq{\geqslant}

\renewcommand\leq{\leqslant}

\newcommand\be{\begin{eqnarray*}}
\newcommand\ee{\end{eqnarray*}}

\newcommand\Q{\mathbb Q}
\newcommand\R{\mathbb R}

\newcommand\N{\mathbb N}

\newcommand\E{\mathbb E}
\renewcommand\P{\mathbb P}

\newcommand\calo{{\mathcal O}}

\newcommand\calf{{\mathcal F}}

\newcommand\eps{\varepsilon}

\newcommand\newop[2]{\def#1{\mathop{\rm #2}\nolimits}}
\newop\log{log}
\newop\ord{ord}
\newop\Gal{Gal}
\newop\SL{SL}
\newop\Bl{Bl}
\newop\mult{mult}
\newop\mass{mass}
\newop\div{div}
\newop\codim{codim}
\newop\sing{sing}
\newop\vdim{vdim}
\newop\edim{edim}
\newop\Ass{Ass}
\newop\size{size}
\newcommand\eqnref[1]{(\ref{#1})}

\newcommand\fbul{\calf_{\bullet}}

\newcommand\vbul{V_{\bullet}}
\newcommand\ybul{Y_{\bullet}}

\newcommand\emax{e_{\max}}

\newcommand\phifbul{\varphi_{\fbul}}
\newcommand\wtilde[1]{\widetilde{#1}}

\newcommand\eqdef{=_{\rm def}}

\newcommand{\deq}{\ensuremath{ \stackrel{\textrm{def}}{=}}}
\newcommand{\equ}{\ensuremath{\,=\,}}
\newcommand{\st}[1]{\ensuremath{ \left\{ #1 \right\} }}
\newcommand\dybul{\Delta_{\ybul}}
\newcommand\calfbul{\calf_{\bullet}}
\newcommand{\HH}[3]{\ensuremath{H^{#1}\left(#2,#3\right)}}
\newcommand{\OO}{\ensuremath{\mathcal O}}
\DeclareMathOperator{\divv}{div}

\newcommand\beginproof[1]{\trivlist\item[\hskip\labelsep{\em #1.}]}
\renewcommand\proof{\beginproof{Proof}}
\newcommand\proofof[1]{\beginproof{Proof of #1}}
\def\endproof{\hspace*{\fill}\endproofsymbol\endtrivlist}

\def\endproofsymbol{\frame{\rule[0pt]{0pt}{6pt}\rule[0pt]{6pt}{0pt}}}

\def\keywordname{{\bfseries Keywords}}%
\def\keywords#1{\par\addvspace\medskipamount{\rightskip=0pt plus1cm
\def\and{\ifhmode\unskip\nobreak\fi\ $\cdot$
}\noindent\keywordname\enspace\ignorespaces#1\par}}
\def\subclassname{{\bfseries Mathematics Subject Classification
(2000)}\enspace}
\def\subclass#1{\par\addvspace\medskipamount{\rightskip=0pt plus1cm
\def\and{\ifhmode\unskip\nobreak\fi\ $\cdot$
}\noindent\subclassname\ignorespaces#1\par}}

\begin{document}

\author{M.~Dumnicki, A.~K\"uronya\footnote{During this project Alex K\"uronya was partially supported by  DFG-Forschergruppe 790 ``Classification of Algebraic Surfaces and Compact Complex Manifolds'',
and the OTKA Grants 77476 and  81203 by the Hungarian Academy of Sciences.}, C.~Maclean, T.~Szemberg\footnote{Szemberg research was partially supported
by NCN grant UMO-2011/01/B/ST1/04875}}
\title{Seshadri constants via Okounkov functions and the   Segre-Harbourne-Gimigliano-Hirschowitz Conjecture}
\date{\today}
\maketitle
\thispagestyle{empty}

\begin{abstract}
   In this paper we relate the SHGH Conjecture to the rationality of one-point Seshadri constants
   on blow ups of the projective plane, and explain how rationality
   of Seshadri constants can be tested with the help of functions on Newton--Okounkov bodies.
\keywords{Nagata Conjecture, SHGH Conjecture, Seshadri constants, Okounkov bodies}
\subclass{MSC 14C20}
\end{abstract}


\section{Introduction}
   Nagata's conjecture and its generalizations have been a central problem in the theory of surfaces for many years, and much work
   has been done towards verifying them \cite{Nag59}, \cite{Cil00}, \cite{Har01}, \cite{StrSze04}, \cite{CHMR12}. In this paper we open
a new line of attack in which we relate Nagata-type statements
   to the rationality of one-point Seshadri constants and invariants of functions on Newton--Okounkov bodies. We obtain as a consequence of our approach
   some evidence that certain Nagata-type questions might be false.

   Seshadri constants were first introduced by Demailly in the course of his work on Fujita's conjecture \cite{Dem92}
   in the late 80's and have been the object of considerable interest ever
since.
   Recall that given a smooth projective variety $X$ and a nef line bundle $L$ on
   $X$, the Seshadri constant of $L$
   at a point $x\in X$ is the real number
   \begin{equation}\label{eq:sesh def}
      \eps(L;x)\eqdef\inf_C\frac{L\cdot C}{\mult_xC} \ ,
   \end{equation}
   where the infimum is taken over all irreducible
   curves passing through $x$.
An intriguing and notoriously difficult problem about Seshadri constants
on surfaces is the question whether these invariants are rational numbers,
see \cite[Remark 5.1.13]{PAG}
It follows quickly from their definition that if a Seshadri constant is
irrational then it must be $\eps(L;x)=\sqrt{L^2}$, see e.g.
\cite[Theorem 2.1.5]{PSC}. It is also known that Seshadri constants of a fixed
line bundle $L$, take their maximal value on a subset in $X$ which is a
complement of at most countably many Zariski closed proper subsets of $X$.

We denote this maximum by $\eps(L,1)$. Similar notation is used for multi-point
Seshadri constants, see \cite[Definition 1.9]{PSC}.
In particular, if $\eps(L;x)=\sqrt{L^2}$ at some point the same
holds in a  very general point on $X$ and $\eps(L;1)=\sqrt{L^2}$.

   From a slightly different point of view, Seshadri constants reveal
information on the structure of the nef cone on the blow-up of $X$ at $x$,
   hence their study is closely related to our attempts
to understand Mori cones of surfaces.

   An even older problem concerning linear series on algebraic
   surfaces is the conjecture formulated by Beniamino Segre in 1961
   and rediscovered, made more precise and reformulated by Harbourne 1986,
   Gimigliano 1987 and Hirschowitz 1988. (See \cite{Cil00} for a very nice
   account on this development and related subjects.) In particular it is known,
   \cite[Remark 5.12]{Cil00} that the
SHGH Conjecture implies the Nagata Conjecture.
   We now recall this conjecture, using by Gimigliano's formuation,
which will be the most convenient form for us \cite[Conjecture 3.3]{Gim89}.
\begin{varthm*}{SHGH Conjecture}
   Let $X$ be the blow up of the projective plane $P^2$ in $s$ general points
   with exceptional divisors $E_1,\dots,E_s$. Let $H$ denote the pullback
   to $X$ of the hyperplane bundle $\calo_{\P^2}(1)$ on $\P^2$. Let
   the integers $d,m_1\geq \dots \geq m_s\geq -1$ with $d\geq m_1+m_2+m_3$
   be given. Then the line bundle
   $$dH-\sum_{i=1}^s m_iE_i$$
   is non-special.
\end{varthm*}
   The main result of this note is the following somewhat unexpected relation
   between the SHGH Conjecture and the rationality problem for Seshadri constants.
\begin{varthm}{Theorem}\label{thm:main}
   Let $s\geq 9$ be an integer for which the SHGH Conjecture holds true.
   Let $X$ be the blow up of the projective plane $\P^2$ in $s$ general
   points. Then
   \begin{itemize}
      \item[a)] either there exists on $X$ an ample line bundle whose Seshadri constant
      at a very general point is irrational;
      \item[b)] or the SHGH Conjecture fails for $s+1$ points.
   \end{itemize}
\end{varthm}
   Note that it is known that the SHGH conjecture holds true for $s\leq 9$, \cite[Theorem 5.1]{Cil00}.
   It is also known that Seshadri constants of ample line bundles on del Pezzo surfaces
   (i.e. for $s\leq 8$) are rational, see \cite[Theorem 1.6]{San}. In any case, the
   statement of the Theorem is interesting (and non-empty) for $s=9$.
   (See the challenge at the end of the article.)

\begin{corollary}
 If all one-point Seshadri constants on the blow-up of $\P^2$ in nine general points are rational, then the SHGH conjecture fails for ten points.
\end{corollary}

   An interesting feature of our proof is that the role played by the general position of the points at which we blow up becomes clear.

   In a different direction, we study the connection between functions on
Newton--Okounkov bodies defined by orders of vanishing, and Seshadri-type
   invariants. Our main result along these lines is the following.

   \begin{varthm}{Theorem}\label{thm:max and rtl}
    Let $X$ be a smooth projective surface, $\ybul$ an admissible flag,  $L$ a big line bundle on $X$, and let $P\in X$ be an arbitrary point.
    \[
    \text{If\ \ } \max_{x\in\Delta_{\ybul}(L)} \phi_{\ord_p}(x) \in \Q\ \ \text{then}\ \ \eps(L,P) \in \Q\ .
    \]

   \end{varthm}

\paragraph*{Acknowledgements.}  We thank Cristiano Bocci, S\'ebastien Boucksom and Patrick Graf for helpful discussions.
Part of this work was done while the second author
was visiting the Uniwersytet Pedagogiczny in Cracow. We would like to thank
the Uniwersytet Pedagogiczny for the excellent working conditions.

\section{Rationality of one point Seshadri constants and the SHGH Conjecture}
In this section we prove Theorem \ref{thm:main}: we start with notation and
preliminary leamms.
   Let $f:X\to\P^2$ be the blow up of $\P^2$ at $s\geq 9$ general points
   $P_1,\dots,P_s$
   with exceptional divisors $E_1,\ldots,E_s$.
   We denote as usual by $H=f^*(\calo_{\P^2}(1))$ the pull back
   of the hyperplane bundle and we let
   $\E=E_1+\dots+E_s$ be the sum of exceptional divisors.
We consider the blow up $g:Y\to X$ of $X$ at $P$ with exceptional divisor $F$.
   Whilst The following result is
   well known, we include it for the lack of a proper reference.
\begin{lemma}\label{lem:almost homog div}
   If there exists a curve $C\subset X$ in the linear system $dH-\sum_{i=1}^sm_iE_i$
   computing Seshadri constant of a $\Q$-line bundle $L=H-\alpha\E$, then
   there exists a divisor $\Gamma$ with $\mult_{P_1}\Gamma=\ldots=\mult_{P_s}\Gamma=M$
   computing Seshadri constant of $L$ at $P$, i.e.
   $$\frac{L\cdot\Gamma}{\mult_P\Gamma}=\frac{L\cdot C}{\mult_PC}=\eps(L;P).$$
\end{lemma}
\proof
   Since the points $P_1,\dots,P_s$ are general, there exist curves
   $$C_{\sigma}=dH-\sum_{j=1}^s m_{\sigma(j)}E_j$$
   for all permutations $\sigma\in\Sigma_s$. Since the point $P$ is general,
we may take all these curves to have the same multiplicity $m$ at $P$.
   Summing over a cycle $\sigma$
   of length $s$ in $\Sigma_s$, we
   obtain a divisor
   $$\Gamma=\sum\limits_{i=1}^s C_{\sigma^i}=sdH-\sum\limits_{i=1}^s \sum\limits_{j=1}^s m_{\sigma^i(j)}E_j=sd H- M\E,$$
   with $M=m_1+\ldots+m_s$.
   Note that the multiplicity of $\Gamma$ at $P$ equals $sm$.
   Taking the Seshadri quotient for $\Gamma$ we have
   $$\frac{L\cdot \Gamma}{sm}=\frac{sd-\alpha s M}{sm}=\frac{d-\alpha M}{m}=\eps(L;P)$$
   hence $\Gamma$ satisfies the assertions of the Lemma.
\endproof
   The following auxiliary Lemma will be used in the proof
   of Theorem \ref{thm:main}. We postpone its proof to the end of this section.
\begin{lemma}\label{lem:estimate for cremona ok case}
   Let $s \geq 9$ be an integer.
   The function
   \begin{equation}\label{twosides}
      f(\delta)=(2\sqrt{s+1}-s)\sqrt{1-s\delta^2}+s(1-\sqrt{s+1})\delta+s-2
   \end{equation}
   is non-negative for $\delta$ satisfying
   \begin{equation}\label{eq:cond on delta}
      \frac{1}{\sqrt{s+1}}<\delta<\frac{1}{\sqrt{s}}.
   \end{equation}
\end{lemma}
\proofof{Theorem \ref{thm:main}}
   Let $\delta$ be a rational number satisfying \eqnref{eq:cond on delta}.
   Note that the SHGH Conjecture implies the Nagata Conjecture \cite[Remark 5.12]{Cil00}
   so that
   $$\eps(\calo_{\P^2}(1);s)=\frac1{\sqrt{s}}$$
   and hence the $\Q$--divisor $L=H-\delta\E$ is ample.
   If $\eps(L;1)$ is irrational, then we are done.

   So we proceed assuming that $\eps(L;1)$ is rational
   and that it is not equal to $\sqrt{L^2}$ (this can be achieved
   changing $\delta$ a little bit if necessary).
   In particular, by Lemma \ref{lem:almost homog div}
for a general point $P\in X$ Seshadri constant $\eps(L;P)$  there is
   a divisor $\Gamma\subset\P^2$ of degree $\gamma$ with
$M=\mult_{P_1}\Gamma=\ldots=\mult_{P_s}\gamma$
   and $m=\mult_{P}\Gamma$
   whose proper transform $\wtilde{\Gamma}$ on $X$ computes the Seshadri
   constant
   $$\eps(L;P)=\frac{L\cdot \wtilde{\Gamma}}{m}=\frac{\gamma - \delta s M}{m}<\sqrt{1-s\delta^2}.$$
   This gives an upper bound on $\gamma$
   \begin{equation}\label{eq:upper bound on gamma}
      \gamma<m\sqrt{1-s\delta^2}+\delta s M.
   \end{equation}
   We need to prove that statement b) in Theorem \ref{thm:main} holds.
Suppose not: the SHGH Conjecture then holds for $s+1$ points in $\P^2$.
The Nagata Conjecture then also holds for $s+1$ points and this gives
   a lower bound for $\gamma$, since for $\Gamma$ we must have that
   \begin{equation}\label{eq:lower bound on gamma}
      \frac{\gamma}{sM+m}\geq\frac{1}{\sqrt{s+1}}.
   \end{equation}
We now claim that
   \begin{equation}\label{eq:no cremona}
      \gamma\geq 2M+m.
   \end{equation}
Suppose not. We then have that
   \begin{equation}\label{eq:no no cremona}
      \gamma< 2M+m.
   \end{equation}
   The real numbers
   $$a:=\frac{2\sqrt{s+1}-s}{2-\delta s}\;\mbox{ and }\; b:=\frac{s-\delta s\sqrt{s+1}}{2-\delta s}$$
   are positive. Multiplying \eqnref{eq:upper bound on gamma} by $a$ and \eqnref{eq:no no cremona} by $b$
   and adding we obtain
   $$sM+m\leq \gamma\sqrt{s+1}< sM+(b+a\sqrt{1-s\delta^2})m,$$
   where the first inequality follows from \eqnref{eq:lower bound on gamma}.
   Subtracting $sM$ in the left and in the right term and dividing by $m$
   we obtain
   $$1<b+a\sqrt{1-s\delta^2}.$$
   Plugging in the definition of $a$ and $b$
   and rearranging terms we obtain that
   $$(2\sqrt{s+1}-s)\sqrt{1-s\delta^2}+s-\delta s\sqrt{s+1}<2-\delta s,$$
   which contradicts Lemma \ref{lem:estimate for cremona ok case}.
   Hence \eqnref{eq:no cremona} holds.

   It follows now from the SHGH conjecture
   for $s+1$ points (in the form stated in the introduction) that
   the linear system
   $$\gamma H-M\E-m F$$
   on $Y$ is non-special.
   Indeed the condition
   $\gamma\geq 2M+m$ is \eqnref{eq:no cremona} and the condition $\gamma\geq 3M$ is satisfied since $\frac{\gamma}{sM}>\frac1{\sqrt{s}}$
   (because the Nagata Conjecture holds for $s$ by hypothesis)
   and because we have assumed that $s\geq 9$.
   This system is also non-empty because the proper transform of $\Gamma$
   under $g$ is its member. Thus by a standard dimension count
   $$0\leq \gamma(\gamma+3)-sM(M+1)-m(m+1).$$
   The upper bound on $\gamma$ \eqnref{eq:upper bound on gamma}
   together with the above inequality yields
   \begin{equation}\label{eq:upper bound on dim C}
      0\leq
     (s\delta M +m\sqrt{1-s\delta^2})(s \delta M +m\sqrt{1-s\delta^2}+3)-m^2-m-sM-sM^2.
   \end{equation}
   Note that the quadratic term in \eqnref{eq:upper bound on dim C} is
   a negative semi-definite form
   $$(s^2\delta^2-s)M^2+2s\delta\sqrt{1-s\delta^2}Mm-s\delta^2m^2.$$
   Indeed, the restrictions on $\delta$ made in \eqnref{eq:cond on delta}
   imply that the term at $M^2$ is negative. The determinant
   of the associated symmetric matrix vanishes. These two conditions
   imply together that the form is negative semi-definite. In particular
   this term of \eqnref{eq:upper bound on dim C} is non-positive.
   The linear part in turn is
   $$(3s\delta-s)M+(3\sqrt{1-s\delta^2}-1)m,$$
   which is easily seen to be negative. This provides the desired contradiction
   and finishes the proof of the Theorem.
\endproof

\begin{remark}
As it is well known, Nagata's conjecture can be interpreted in terms of the nef and Mori cones of the blow-up $X$ of $\P^2$ at $s$ general points. More
precisely, consider the following question: for what $t\geq 0$ does
the ray $H-t\E$ meet the boundary of the nef cone?
The conjecture predicts that this ray should intersect the boundaries
of the nef cone and the effective cone at the same time.
\end{remark}
Considering  the Zariski chamber structure of $X$ (see \cite{BKS04}), we see
that this is equivalent to requiring that $H-t\E$ crosses
exactly one Zariski chamber (the
nef cone itself). Surprisingly, it is easy to prove that
$H-t\E$ cannot cross more than two chambers.
\begin{proposition}\label{prop:only 2 chambers}
   Let $f:X\to\P^2$ be the blow up of $P^2$ in $s$ general points with
exceptional divisors $E_1,\dots,E_s$.
   Let $H$ be the pull-back of the hyperplane bundle and $\E=E_1+\ldots+E_s$.
   The ray $R=H-t\E$ meets at most two Zariski chambers on $X$.
\end{proposition}
\proof
   If $\eps=\eps(\calo_{\P^2}(1);s)=\frac1{\sqrt{s}}$ i.e. this multi-point Seshadri constant
   is maximal, then the ray crosses only the nef cone.

   If $\eps$ is submaximal, then there is a curve $C=dH-\sum m_iE_i$
   computing this Seshadri constant, i.e. $\eps=\frac{d}{\sum m_i}$.

   If this curve is homogeneous,
   i.e. $m=m_1=\dots=m_s$, then we claim first that $\mu=\mu(\calo_{\P^2}(1);\E)=m/d$.
   Indeed, this is an effective divisor on the ray $R$
   and it is not big (because big divisors on surfaces intersect
   all nef divisors positively, see \cite[Corollary 3.3]{BauSch12}),
   so it must be the point where the ray leaves the big cone.

   Now, suppose that for some $\eps<\delta<\mu$ the ray $R$ crosses
   another Zariski chamber wall. This means that there is a divisor $eH-k\E$
   (obtained after possible symmetrization of a curve $D$ with $0$ intersection
   number with $H-\delta\E$) with
   $$(H-\delta\E)\cdot(eH-k\E)=0.$$
   Hence $e=\delta ks<\mu ks$. On the other hand
   $$(eH-k\E)\cdot (dH-\mu\E)=e-k\mu s<0,$$
   implies that $C$ is a component $eH-k\E$ which is not possible.
   Hence there are only two Zariski chambers meeting the ray $R$
   in this case.

   If the curve $C$ is not homogeneous, then since the points are general
   there exist at least (and also at most) $s$ different irreducible
   curves computing $\eps$. All these curves are in the support
   of the negative part of the Zariski decomposition of $H-\lambda\E$
   for $\lambda>\eps$. Hence their intersection matrix
   is negative definite and this is a matrix of maximal dimension (namely $s$)
   with this property. This implies that $R$ cannot meet another Zariski
   chamber because the support of the negative part of Zariski decompositions
   grows only when encountering new chambers, see \cite{BKS04}.
\endproof

It is interesting to compare this result with the following easy example,
which constructs rays meeting a maximal number of chambers.
\begin{example}
   Keeping the notation from Proposition \ref{prop:only 2 chambers}
   let $L=(\frac{s(s+1)}{2}+1)H-E_1-2E_2-\ldots-sE_s$ is an ample divisor
   on $X$ and the ray $R=L+\lambda\E$ crosses $s+1=\rho(X)$ Zariski chambers.
   Indeed, with $\lambda$ growing, exceptional divisors $E_1, E_2,\ldots, E_s$
   join the support of the Zariski decomposition of $L-\lambda\E$ one by one.
   We leave the details to the reader.
\end{example}

   We conclude this section with the proof of Lemma
\ref{lem:estimate for cremona ok case}.

\proofof{Lemma \ref{lem:estimate for cremona ok case}}
   Since $f(1/\sqrt{s+1})=0$ it is enough to show
   that $f(\delta)$ is increasing for $1/\sqrt{s+1}\leq\delta\leq 1/\sqrt{s}$.
   Consider the derivative
   \begin{equation}\label{eq:pochodna f}
      f'(\delta)=s\left(1+\frac{\delta}{\sqrt{1-s\delta^2}}(s-2\sqrt{s+1})-\sqrt{s+1}\right).
   \end{equation}
   The function $h(\delta)=\frac{\delta}{\sqrt{1-s\delta^2}}$
   is increasing for $1/\sqrt{s+1}\leq\delta\leq 1/\sqrt{s}$
   since the numerator is an increasing function of $\delta$ and the
denominator is a decreasing function of $\delta$.
   We have $h(\frac{1}{\sqrt{s+1}})=1$ so that $h(\delta)\geq 1$ holds for all $\delta$.
   Since the coefficient at $h(\delta)$ in \eqnref{eq:pochodna f} is positive
   we have
   $$f'(\delta)\geq s\left(1+(s-2\sqrt{s+1})-\sqrt{s+1}\right)=1+s-3\sqrt{s+1}>0,$$
   which completes the proof.
\endproof

\section{Rationality of Seshadri constants and functions on Okounkov bodies}
   The theory of Newton--Okounkov bodies has emerged recently with work by
 Okounkov
   \cite{Ok96}, Kaveh--Khovanskii \cite{KavKh08}, and Lazarsfeld--Musta\c t\u a \cite{LazMus09}.
   Shortly thereafter, Boucksom--Chen \cite{BouChe11} and Witt-Nystr\"om \cite{Nys09} have shown ways of constructing geometrically significant
   functions on Okounkov bodies, that were further studied in \cite{KMS}.
   In the context of this note the study of Okounkov functions
   was pursued by the last three authors in \cite{KMS}. We refer to \cite{KMS}
   for construction and properties of Okounkov functions.

   In this section we consider an arbitrary smooth projective surface
$X$ and an ample line bundle $L$ on $X$.
   Let $p\in X$ be an arbitrary point and let $\pi:Y\to X$ be the
   blow up of $p$ with exceptional divisor $E$. Recall that
   the Seshadri constant of $L$ at $p$ can equivalently be defined as
\[
   \eps(L;p) = \sup\st{t>0\,|\, \pi^*L-tE \text{ is nef }}\ .
\]
   There is a related invariant
\[
   \mu(L;p) \deq \sup\st{t>0\,|\, \pi^*L-tE \text{ is pseudo-effective}} \equ \sup\st{t>0\,|\, \pi^*L-tE \text{ is big}} \ .
\]
The invariant $\eps(L;p)$ is the value of the parameter $\lambda$
where the ray $\pi^*L-\lambda E$ meets the boundary of the nef cone of $Y$,
and $\mu(L;p)$ is the value of $\lambda$ where the ray
meets the boundary of the pseudo-effective cone. The following relation
   between the two invariants is important in our considerations.
\begin{remark}\label{rmk:KLM}
   If $\epsilon(L;p)$ is irrational, then
\[
   \epsilon(L;p) \equ \mu(L;p)\ .
\]
   In particular, if $\mu(L;p)$ is rational, then so is $\epsilon(L;p)$.
\end{remark}
   Rationality of $\mu(L;p)$ implies rationality of the associated Seshadri
constants on surfaces. This invariant appears in the study of the concave
function $\varphi_{\ord_p}$ associated to the geometric valuation on $X$ defined
by the order of vanishing $\ord_p$ at $p$.
   We fix some flag $\ybul: X\supseteq C\supseteq \st{x_0}$ and consider
   the Okounkov body $\dybul(L)$ defined with respect to that flag.
   We define also a multiplicative filtration
   determined by the geometrical valuation $\ord_P$
   on the graded algebra $V=\oplus_{k\geq 0}V_k$
   with $V_k=H^0(X,kL)$
   by
   $$\calf_t(V)=\left\{s\in V:\; \ord_P(s)\geq t\right\},$$
   see \cite[Example 3.7]{KMS} for details. (All the above remains valid
   in the more general context of graded linear series.)
   There is an induced filtration $\fbul(V_k)$ on every summand of $V$
   and one defines the maximal jumping numbers of both filtrations as
   $$e_{\max}(V,\fbul)=\sup\left\{t\in\R:\; \exists k \calf_{kt}V_k\neq 0\right\}
     \;\mbox{ and }\;
     e_{\max}(V_k,\fbul)=\sup\left\{t\in\R:\; \calf_tV_k\neq 0\right\}$$
   respectively.
   Let $\varphi_{\ord_P}(x)=\phifbul(x)$ be
   the Okounkov function on $\dybul(L)$ determined by filtration $\fbul$,
   see \cite[Definition 4.8]{KMS}. It turns out that
   $\mu(L;p)$ is the maximum of the Okounkov function $\varphi_{\ord_P}$.
\begin{proposition}\label{prop:equiv}
 With notation as above we have that
\[
 \mu(L;p) \equ \limsup_{m\to\infty}\frac{\max\st{\ord_p(s)\,|\, s\in\HH{0}{X}{\OO_X(mL)}}}{m} \equ \max_{x\in\Delta_{Y_\bullet}(L)}\phi_{\ord_p}(x)\ .
\]
\end{proposition}
\begin{proof}
 Observe that
\[
 \ord_p (s) \equ \ord_E(\pi^*s) \equ \max\st{m\in\N\,|\, \divv(\pi^*s)-mE \text{ is effective}}\ .
\]
Consequently,
\begin{eqnarray*}
  \mu(L;p) &  = &  \sup\st{t\in\R_{\geq 0}\,|\, \pi^*L-tE \text{ is pseudo-effective}} \\
& = & \limsup_{m\to\infty} \frac{\max\st{\ord_p(s)\,|\,  s\in\HH{0}{X}{\OO_X(mL}}}{m}\ ,
\end{eqnarray*}
which gives the first equality.

For the second equality, we observe first that
\[
\max\st{\ord_p(s)\,|\, s\in\HH{0}{X}{\OO_X(mL)}} \equ e_{\max}(V_m,\calfbul)\ ,
\]
and hence
\[
\limsup_{m\to\infty}\frac{\max\st{\ord_p(s)\,|\, s\in\HH{0}{X}{\OO_X(mL)}}}{m} \equ e_{\max}(V,\calfbul)\ .
\]
Since
\[
 e_{\max}(V,\calfbul) \equ \max_{x\in\Delta_{Y_\bullet}(L)}\varphi_{\ord_p}(x)\
\]
by Theorem~\ref{thm:maximum}, we are done.
\end{proof}
\subsection{Independence of the maximum of an Okounkov function on the flag}

In the course of this section the projective variety $X$ can have arbitrary dimension.

   Boucksom and Chen proved that
   though $\phifbul$ and $\Delta(\vbul)$ depend on the flag $\ybul$,
   the integral of $\phifbul$ over $\Delta(\vbul)$ is independent of $\ybul$,
   \cite[Remark 1.12 (ii)]{BouChe11}. We prove now
   that the maximum of the Okounkov function does not depend on the flag.
   This fact is valid in the general setting of arbitrary
   multiplicative filtration $\calf$ defined on a graded linear series $\vbul$.

\begin{remark}
Note that in general the functions $\phifbul$ are only upper-semicontinuous and concave, but not continuous on the whole Newton--Okounkov body
as explained in \cite[Theorem 1.1]{KMS}. They are however continuous provided the underlying body  $\Delta(\vbul)$ is a polytope (see again \cite[Theorem 1.1]{KMS}),
which is the case for complete linear series on surfaces \cite{KLM}.
\end{remark}

\begin{theorem}[Maximum of Okounkov functions]\label{thm:maximum}
   With the above notation, we have that
   $$\max\limits_{x\in\Delta_{\ybul}(L)}\phifbul(x)=\emax(L,\fbul).$$
   In particular the left hand side does not depend on the flag $\ybul$.
\end{theorem}
\proof
 For any real $t\geq 0$,
we consider the partial Okounkov body
$\Delta_{t,\ybul}(L)$  associated the graded  linear series $V_{t,k}\subset H^0(kL)$ given by
 \[
 V_{t,k} \deq  \calf_{kt}(H^0(kL))\ .
 \]
Note that by definition
\[ \emax(L,\fbul)= \sup\{t\in \R| \cup_k V_{t,k}\neq 0.\}\]
In other words,
\[ \emax(L, \fbul)= \sup\{t\in \R| \Delta_{t,\ybul}(L)\neq \emptyset\}.\]
Recall that by definition
\[ \phifbul(x)=\sup\{ t\in \R| x\in \Delta_{t,\ybul}(L)\}.\]
and it is therefore immediate that $\forall x$
\[ \phifbul(x)\leq \emax(L, \fbul).\]
from which it follows that
\[ \max\limits_{x\in\Delta_{\ybul}(L)}\phifbul(x)\leq \emax(L, \fbul).\]
Since the bodies $\Delta_{t,\ybul}(L)$ form a decreasing family of closed subsets
of $\R^d$, we have that
\[ \cap_{t| \Delta_{t,\ybul}(L)\neq \emptyset} \Delta_{t,\ybul}(L)\neq \emptyset.\]
Consider a point $y\in \cap_{t| \Delta_{t,\ybul}(L)\neq \emptyset} \Delta_{t,\ybul}(L)$
We then have that
\[ y\in \Delta_{t,\ybul}(L)\Leftrightarrow \Delta_{t,\ybul}(L)\neq \emptyset\]
and hence
\[ \sup\{ t\in \R| y\in \Delta_{t,\ybul}(L)\}=
\sup\{t\in \R| \Delta_{t,\ybul}(L)\neq \emptyset\}\]
or in other words
\[\phifbul(y)=\emax(L, \fbul)\]
from which it follows that
\[ \max\limits_{x\in\Delta_{\ybul}(L)}\phifbul(x)\leq \emax(L, \fbul).\]
This completes the proof of the theorem.
\endproof

\section{The effect of blowing up on Okounkov bodies and functions}
   We begin with an observation (valid in fact in arbitrary dimension, though
we state and prove it here only for surfaces.)

\begin{proposition}\label{prop:restricted body and ftion}
   Let $S$ be an arbitrary surface
   with a fixed flag $\ybul$ and let $f:X\to S$ be the blow up of $S$ at a point $P$
   not contained in the divisorial part of the flag, $P\notin Y_1$. Let $E$ be the
   exceptional divisor. And finally let $D$ be a big divisor on $S$.
   For any rational number $\lambda$ such that
   $0\leq \lambda < \mu(L;P)$ we let $D_{\lambda}$ be the $\Q$--divisor
   $f^*D-\lambda E$. There is then a natural
   inclusion
   $$\Delta_{\ybul}(D_{\lambda})\subset \Delta_{\ybul}(D).$$
   Moreover, a filtration $\fbul$ on the graded algebra $\oplus_{k\geq 0}H^0(S;kD)$
   induces a filtration $\fbul^{\lambda}$ on the graded (sub)algebra $\oplus H^0(X,kD_{\lambda})$,
   where the sum is taken over all $k$ divisible enough.
   For associated Okounkov functions we have
   \begin{equation}\label{eq:ok ftions compare}
      \varphi_{\fbul^{\lambda}}(x)\leq \varphi_{\fbul}(x)
   \end{equation}
   for all $x\in\dybul(D_{\lambda})$.
\end{proposition}

\begin{remark}
The best case scenario is that the functions $\phi$ are piecewise linear with
rational coefficients
over a rational polytope. Of these properties, some evidence for the
first was given by Donaldson \cite{Don1} in the toric situation.
For the second condition, it was proven in \cite{AKL} that every line bundle
on a surface has an Okounkov body which is a rational polytope.
\end{remark}

\proof
Note first that since the blow up center is disjoint
   from all elements in the flag, one can take $\ybul$ to be
an admissible flag on $X$.
(Strictly speaking one takes $f^*\ybul$ as the flag, but it should cause no
   confusion to identify flag elements upstairs and downstairs.)

   Then, if $k$ is sufficiently divisible we have
   $$H^0(X,f^*kD-k\lambda E)=\left\{s\in H^0(S,kD):\; \ord_P(s)\geq E\right\}\subset H^0(S,kD).$$
The inclusion of the Okounkov bodies follows immediately under this
identification. We can view the algebra associated to $f^*D-E$ as a graded
linear series on $S$.

   The claim about the Okounkov functions follows from their definition,
see \cite[Definition 4.8]{KMS}. Indeed,  the supremum arising in the
   definition of $\varphi_{\fbul^{\lambda}}$ is taken
   over a smaller set of sections than it is for $\varphi_{\fbul}$.
\endproof
   The following examples illustrate various situations arising in the
   setting of Proposition \ref{prop:restricted body and ftion}.
\begin{example}\label{ex:p2}\rm
   Let $\ell$ be a line in $X_0=\P^2$ and let $P_0\in\ell$ be a point.
   We fix the flag
   $$\ybul:\; X_0\supset\ell\supset\left\{P_0\right\}.$$
   Let $D=\calo_{\P^2}(1)$. Then $\dybul(D)$ is simply the standard simplex in $\R^2$.
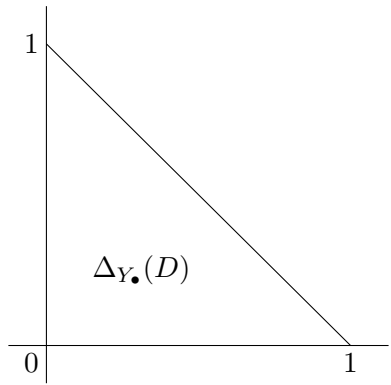
\unitlength.1mm
\begin{center}
\begin{pgfpicture}{0cm}{0cm}{6cm}{6cm}
   \pgfline{\pgfxy(0.5,1)}{\pgfxy(5.5,1)}
   \pgfline{\pgfxy(1,0.5)}{\pgfxy(1,5.5)}
   \pgfline{\pgfxy(5,1)}{\pgfxy(1,5)}
   \pgfputat{\pgfxy(5,0.9)}{\pgfbox[center,top]{1}}
   \pgfputat{\pgfxy(0.9,0.9)}{\pgfbox[right,top]{0}}
   \pgfputat{\pgfxy(0.9,5)}{\pgfbox[right,center]{1}}
   \pgfputat{\pgfxy(2.9,2)}{\pgfbox[right,center]{$\dybul(D)$}}
\end{pgfpicture}
\end{center}
   Let $\fbul$ be the filtration on the complete linear series of $D$
   imposed by the geometric valuation $\nu=\ord_{P_0}$ and let
   $\varphi_{\nu}$ be the associated Okounkov function.
   Then
   $$\varphi_{\nu}(a,b)=a+b.$$
   Indeed, given a point $(a,b)$ with rational coordinates, we
   pass to the integral point $(ka,kb)$. This valuation vector
   can be realized geometrically by a global section in $H^0(\P^2,\calo_{\P^2}(k))$
   vanishing exactly with multiplicity $ka$ along $\ell$, exactly
   with multiplicity $kb$ along a line passing through $P_0$ different from $\ell$
   and along a curve of degree $k(1-a-b)$ not passing through $P_0$.
\end{example}
   The next example shows that even when the Okounkov body changes in the course
   of blowing up, the Okounkov function may remain the same.
\begin{example}\label{ex:ps blown up in 1 point}\rm
   Keeping the notation from the previous Example and from Proposition \ref{prop:restricted body and ftion}
   let $f:X_1=\Bl_{P_1}\P^2\to X_0=\P^2$ be the blow up of the projective plane
   in a point $P_1$ not contained in the flag line $\ell$ with the exceptional divisor $E_1$.
   We work now
   with a $\Q$--divisor $D_{\lambda}=f^*(\calo_{\P^2}(1))-\lambda E_1=H-\lambda E_1$,
   for some fixed $\lambda\in [0,1]$. A direct computation using \cite[Theorem 6.2]{LazMus09} gives that the Okounkov body has the shape
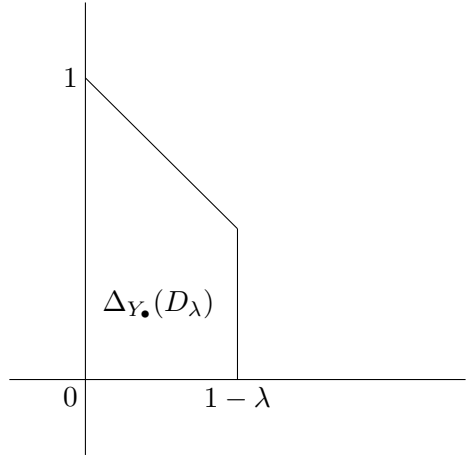
\begin{center}
\begin{pgfpicture}{0cm}{0cm}{6cm}{6cm}
   \pgfline{\pgfxy(0,1)}{\pgfxy(6,1)}
   \pgfline{\pgfxy(1,0)}{\pgfxy(1,6)}
   \pgfline{\pgfxy(1,5)}{\pgfxy(3,3)}
   \pgfline{\pgfxy(3,1)}{\pgfxy(3,3)}
   \pgfputat{\pgfxy(3,0.9)}{\pgfbox[center,top]{$1-\lambda$}}
   \pgfputat{\pgfxy(0.9,0.9)}{\pgfbox[right,top]{$0$}}
   \pgfputat{\pgfxy(0.9,5)}{\pgfbox[right,center]{$1$}}
   \pgfputat{\pgfxy(2.7,2)}{\pgfbox[right,center]{$\dybul(D_{\lambda})$}}
\end{pgfpicture}
\end{center}
   Thus we see that the Okounkov body of $D_{\lambda}$
   is obtained from that of $D$ by intersecting with a closed halfspace.

   For the valuation $\nu=\ord_{P_0}$, we get as above
   $$\varphi_{\nu}(a,b)=a+b.$$
   Let now $k$ be an integer such that the point $(ka,kb)$ is integral and $k\lambda$ is also an integer.
   Now we need to exhibit a section $s$ in $H^0(\P^2,\calo_{\P^2}(k))$
   satisfying the following conditions:
   \begin{itemize}
      \item[a)] $s$ vanishes along $\ell$ exactly to order $a$;
      \item[b)] $s$ vanishes in the point $P_1$ to order at least $k\lambda$;
      \item[c)] $s$ vanishes in the point $P_0$ exactly to order $b$.
   \end{itemize}
   We let the divisor of $s$ to consist of $a$ copies of $\ell$ (there is no other
   choice here), of $b$ copies of the line through $P_0$ and $P_1$,
   of $k\lambda-b$ copies of any other line passing through $P_1$ (if this
   number is negative then this condition is empty) and of a curve of degree $k(1-a-\max\left\{b,\lambda\right\})$
   passing neither through $P_0$, nor through $P_1$.
\end{example}
\begin{remark}
   Note that in the setting of Proposition \ref{prop:restricted body and ftion}
   Okounkov bodies of divisors $D_{\lambda}$ always result from those of $D$
   by cutting with finitely many halfplanes. This is an immediate consequence
   of \cite[Theorem 5]{KLM}.
\end{remark}
   We conclude by showing that the inequality in \eqnref{eq:ok ftions compare}
   can be sharp, i.e. the blow up process can influence the Okounkov function
as well as the Okounkov body.
\begin{example}
   Keeping the notation from the previous examples, let
   $f:X_6\to \P^2$ be the blow up of six general points $P_1,\dots,P_6$
   not contained in $\ell$ and chosen so that the points $P_0,P_1,\dots,P_6$
   are also general. Let $E_1,\dots,E_6$ denote the exceptional divisors
   and set $\E=E_1+\ldots+E_6$. We consider the divisor $D=H-\frac25\E$.
   A direct computation using \cite[Theorem 6.2]{LazMus09}
   (this requires computing Zariski decompositions this time, see \cite{Bau09} for an effective approach)
   yields the triangle with vertices at the origin and in points $(0,1)$
   and $(1/25,0)$ as $\dybul(D)$. For the valuation $\nu=\ord_{P_0}$ we get now
   \begin{equation}\label{eq:ok ftion less a plus b}
      \varphi_{\nu}(a,b)\leq 4/15 <a+b
   \end{equation}
   for $(a,b)\in\Omega=\left\{(x,y)\in\R^2:\; x\in[0,11/360)\mbox{ and } b\in (4/15-a,1-25a]\right\}$.
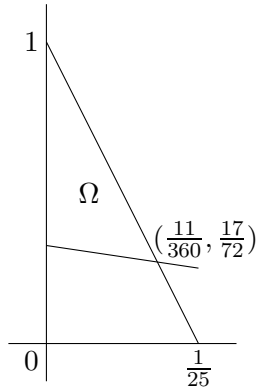
\begin{center}
\begin{pgfpicture}{0cm}{0cm}{6cm}{6cm}
   \pgfline{\pgfxy(0.5,1)}{\pgfxy(3.5,1)}
   \pgfline{\pgfxy(1,0.5)}{\pgfxy(1,5.5)}
   \pgfline{\pgfxy(3,1)}{\pgfxy(1,5)}
   \pgfline{\pgfxy(3,2)}{\pgfxy(1,2.3)}
   \pgfputat{\pgfxy(3,0.9)}{\pgfbox[center,top]{$\frac{1}{25}$}}
   \pgfputat{\pgfxy(2.4,2.15)}{\pgfbox[left,bottom]{$(\frac{11}{360},\frac{17}{72})$}}
   \pgfputat{\pgfxy(0.9,0.9)}{\pgfbox[right,top]{0}}
   \pgfputat{\pgfxy(0.9,5)}{\pgfbox[right,center]{1}}
   \pgfputat{\pgfxy(1.7,3)}{\pgfbox[right,center]{$\Omega$}}
\end{pgfpicture}
\end{center}
   The reason for the above inequality is the following. Let $(a,b)\in\Omega$ be a valuation vector.
   Assume to the contrary
   that $\varphi(a,b)>4/15$. It is well known that $\eps(\calo_{\P^2}(1),P_0,\dots,P_6)=\frac38$,
   see for instance \cite[Example 2.4]{StrSze04}. On the other hand a section
   with the above valuation vector would have (after scaling to $\calo(1)$)
   multiplicities $2/5$ at $P_1,\dots,P_6$ and $\varphi_{\nu}(a,b)>4/15$ at $P_0$.
   It would give Seshadri quotient
   $$\frac{1}{6\cdot\frac25+\varphi_{\nu}(a,b)}<\frac38,$$
   a contradiction. This proves \eqnref{eq:ok ftion less a plus b}.
\end{example}
   Since the SHGH Conjecture holds for $9$ points, the first challenge
   arising in the view of our Theorem would be to compute the Okounkov
   body and the Okounkov function associated to $\ord_{P_0}$ as above
   for the system
   $$22H-7(E_1+\dots+E_9).$$



\bigskip \small

\bigskip
   Marcin Dumnicki,
   Jagiellonian University, Institute of Mathematics, {\L}ojasiewicza 6, PL-30-348 Krak\'ow, Poland

\nopagebreak
   \textit{E-mail address:} \texttt{Marcin.Dumnicki@im.uj.edu.pl}

\bigskip
   Alex K\"uronya,
   Budapest University of Technology and Economics,
   Mathematical Institute, Department of Algebra,
   Pf. 91, H-1521 Budapest, Hungary.

\nopagebreak
   \textit{E-mail address:} \texttt{alex.kuronya@math.bme.hu}

\medskip
   \textit{Current address:}
   Alex K\"uronya,
   Albert-Ludwigs-Universit\"at Freiburg,
   Mathematisches Institut,
   Eckerstra{\ss}e 1,
   D-79104 Freiburg,
   Germany.

\bigskip
   Catriona Maclean,
   Institut Fourier, CNRS UMR 5582   Universit\'e de Grenoble,
   100 rue des Maths,
   F-38402 Saint-Martin d'H\'eres cedex,  France

\nopagebreak
   \textit{E-mail address:} \texttt{catriona@fourier.ujf-grenoble.fr}

\bigskip
   Tomasz Szemberg,
   Instytut Matematyki UP,
   Podchor\c a\.zych 2,
   PL-30-084 Krak\'ow, Poland.

\nopagebreak
   \textit{E-mail address:} \texttt{szemberg@up.krakow.pl}


\end{document}